\documentclass[11pt]{amsart}
\usepackage{amsmath}
\usepackage{amsfonts}
\usepackage{amsthm}
\usepackage{amssymb}
\usepackage{tikz-cd}
\usepackage{graphics}
\usepackage{array}
\usepackage{enumerate}
\usepackage{dsfont}
\usepackage{color}
\usepackage{wasysym}
\usepackage{hyperref}
\usepackage{pdfsync}
\usepackage{mathrsfs}

\newcommand{\N}{{\mathbb N}}

\newcommand{\C}{{\mathbb C}}
\newcommand{\Q}{{\mathbb Q}}

\DeclareMathOperator{\Alb}{Alb}

\DeclareMathOperator{\Pic}{Pic}

\DeclareMathOperator{\codim}{codim}

\DeclareMathOperator{\GV}{GV}
\DeclareMathOperator{\IT}{IT_0}

\DeclareMathOperator{\red}{red}
\DeclareMathOperator{\Var}{Var}
\DeclareMathOperator{\rank}{rank}

\theoremstyle{plain}

\newtheorem{thm}{Theorem}[section]
\newtheorem{conj}[thm]{Conjecture}

\newtheorem{coro}[thm]{Corollary}
\newtheorem{lemma}[thm]{Lemma}

\theoremstyle{definition}

\newtheorem{defi}[thm]{Definition}

\newtheorem{rem}[thm]{Remark}
\newtheorem*{ac}{Acknowledgements}

\begin{document}

\title[Estimates on the Kodaira dimension]{Estimates on the Kodaira dimension for fibrations over abelian varieties}
	
	\author{Fanjun Meng}
	\address{Department of Mathematics, Johns Hopkins University, 3400 N. Charles Street, Baltimore, MD 21218, USA}
	\email{fmeng3@jhu.edu}


\begin{abstract}
We give estimates on the Kodaira dimension for fibrations over abelian varieties, and give some applications. One of the results strengthens the subadditivity of Kodaira dimension of fibrations over abelian varieties.
\end{abstract}

\maketitle
	\setcounter{tocdepth}{1}
	\tableofcontents

\section{Introduction}\label{1}

In this paper, we give estimates on the Kodaira dimension for fibrations over abelian varieties over $\C$, and give some applications.

\begin{thm}\label{hp}
Let $f\colon X \to A$ be a fibration from a smooth projective variety $X$ to an abelian variety $A$ where $f$ is smooth over an open set $V\subseteq A$, and $m$ a positive integer. Then
$$\kappa(V)\ge\kappa(A, \widehat{\det} f_* \omega_{X}^{\otimes m})\ge \dim V^0(A, f_*\omega_X^{\otimes m}).$$
If $m>1$, then $\kappa(A, \widehat{\det} f_* \omega_{X}^{\otimes m})=\dim V^0(A, f_*\omega_X^{\otimes m})$.
\end{thm}

The line bundle $\widehat{\det} f_* \omega_{X}^{\otimes m}$ is the reflexive hull of $\det f_* \omega_{X}^{\otimes m}$. Given a smooth quasi-projective variety $V$, $\kappa(V)$ denotes the log Kodaira dimension, defined as follows: for any smooth projective compactification $Y$ of $V$ such that $D = Y\setminus V$ is a divisor with simple normal crossing support, we have $\kappa (V) = \kappa (Y, K_Y+D)$. Let $\mathcal{F}$ be a coherent sheaf on an abelian variety $A$. The cohomological support locus $V^0(A, \mathcal{F})$ is defined by
$$V^0(A, \mathcal{F})=\{\alpha\in\Pic^0(A)\mid\dim H^0(A, \mathcal{F}\otimes\alpha)>0\},$$
see also Definition \ref{csl}. If $\mathcal{F}$ admits a finite direct sum decomposition
$$\mathcal{F}\cong\bigoplus_{i\in I}(\alpha_i\otimes p_i^*\mathcal{F}_i),$$
where each $A_i$ is an abelian variety, each $p_i\colon A\to A_i$ is a fibration, each $\mathcal{F}_i$ is a nonzero M-regular coherent sheaf on $A_i$, and each $\alpha_i\in\Pic^0(A)$ is a torsion line bundle, then we can characterize $V^0(A, \mathcal{F})$ using this decomposition, and we have
$$\dim V^0(A, \mathcal{F})=\max_{i\in I} \dim A_i.$$
See Definition \ref{GVM} for the definition of M-regular coherent sheaves. We will use this observation in the proofs of our main theorems. The decomposition above is called the Chen--Jiang decomposition of $\mathcal{F}$. It is known that pushforwards of pluricanonical bundles under morphisms to abelian varieties have the Chen--Jiang decomposition by \cite{CJ18, PPS17, LPS20} in increasing generality, and pushforwards of klt pairs under morphisms to abelian varieties have the Chen--Jiang decomposition, as proved independently in \cite{Jia21} and \cite{Men21}.

If there exists a positive integer $m$ such that $\kappa(A, \widehat{\det} f_* \omega_{X}^{\otimes m})=\dim A$, it is known that $\kappa(V)=\dim A$ by \cite[Theorem 2.4]{MP21} (see Theorem \ref{thm:PS3}) which is a consequence of \cite[Theorems 4.1 and 3.5]{PS17}. The first part of the inequality in Theorem \ref{hp} is a generalization of this fact when $\widehat{\det} f_* \omega_{X}^{\otimes m}$ is not necessarily big. In a different direction, by letting $\kappa(V)=0$ in Theorem \ref{hp}, we can recover \cite[Theorem B]{MP21} which gives the structures of pushforwards of pluricanonical bundles of smooth projective varieties under surjective morphisms to abelian varieties when the morphisms are smooth away from a closed set of codimension at least $2$ in the abelian varieties.

By estimating the dimension of $V^0(A, f_*\omega_X^{\otimes m})$, we have the following corollary of Theorem \ref{hp}.

\begin{coro}\label{hpc}
Let $g\colon X\to Y$ be a smooth model of the Iitaka fibration of a smooth projective variety $X$ with general fiber $G$ where $Y$ is a smooth projective variety, and $f\colon X \to A$ a fibration to an abelian variety $A$ where $f$ is smooth over an open set $V\subseteq A$. Then
$$\kappa(V)\ge\dim A-q(G).$$
\end{coro}

Given a projective variety $G$, $q(G)$ denotes the irregularity of $G$, see Definition \ref{irregularity}. A projective variety $G$ is said to be regular if $q(G)=0$. If the Iitaka fibration $g$ has regular general fiber, then $\kappa(V)=\dim A$ by Corollary \ref{hpc} and thus $f$ is not smooth unless $A$ is a point. Thus we have the following corollary.

\begin{coro}\label{nns}
Let $g\colon X\to Y$ be a smooth model of the Iitaka fibration of a smooth projective variety $X$ with general fiber $G$ where $Y$ is a smooth projective variety. If $G$ is regular, then $X$ has no nontrivial smooth morphisms to an abelian variety.
\end{coro}

Corollary \ref{nns} implies that there are no nontrivial smooth morphisms from a projective variety of general type to an abelian variety which was proved in \cite{VZ01} when the base is an elliptic curve and in \cite{HK05} and \cite{PS14a} in general, see also \cite{MP21} for related results.

We also have a quick corollary of Theorem \ref{hp} if the Albanese morphism of $X$ is a fibration. In Corollary \ref{hpc}, we let $f$ be the Albanese morphism of $X$. Then we deduce that
$$\kappa(V)\ge q(Y)$$
by Theorem \ref{hp} and \cite[Theorem D]{LPS20}.

By \cite[Theorem 1.1]{Kaw85a}, Theorem \ref{hp} implies a special case of the Kebekus--Kov\'acs conjecture when the base $V$ compactifies to an abelian variety. This conjecture bounds $\Var(f)$ from above by the log Kodaira dimension $\kappa (V)$ assuming that the general fiber of $f$ has a good minimal model and has recently been proved in \cite{Taj20}. For the definition of the variation $\Var(f)$, see \cite[Section 1]{Kaw85a}. In private communication from Mihnea Popa, he proposed the following conjecture.

\begin{conj}[\cite{Pop22}]
Let $f\colon X \to A$ be a fibration from a smooth projective variety $X$ to an abelian variety $A$. Then
$$\dim V^0(A, f_*\omega_X^{\otimes m})\ge\Var(f)$$
for every integer $m>1$ such that $f_*\omega_X^{\otimes m}\neq0$. 
\end{conj}

In the case when $\kappa(X)=0$ and $f$ is the Albanese morphism of $X$, this conjecture is essentially equivalent to Ueno's Conjecture K, predicting that up to birational equivalence $f$ becomes a projection onto a factor after an \'etale base change. This is due to the fact that $f_*\omega_X^{\otimes m}$ is a torsion line bundle on $A$ for every positive integer $m$ such that $f_*\omega_X^{\otimes m}\neq0$ by \cite[Theorem 5.2]{HPS18}. In the following corollary, we give an answer to his conjecture assuming that the general fiber of $f$ has a good minimal model.

\begin{coro}\label{var}
Let $f\colon X \to A$ be a fibration from a smooth projective variety $X$ to an abelian variety $A$. Assume that the general fiber of $f$ has a good minimal model. Then
$$\dim V^0(A, f_*\omega_X^{\otimes m})\ge\Var(f)$$
for every integer $m>1$ such that $f_*\omega_X^{\otimes m}\neq0$. Moreover, if $g\colon X\to Y$ is a smooth model of the Iitaka fibration of $X$ where $Y$ is a smooth projective variety, then
$$q(Y)\ge\dim V^0(A, f_*\omega_X^{\otimes m})\ge\Var(f)$$
for every integer $m>1$ such that $f_*\omega_X^{\otimes m}\neq0$.
\end{coro}

In a different but related direction, the next theorem strengthens the result on the subadditivity of Kodaira dimension of fibrations over abelian varieties by \cite{CP17} (see also \cite{HPS18}).

\begin{thm}\label{sskd}
Let $f\colon X \to A$ be a fibration from a klt pair $(X, \Delta)$ to an abelian variety $A$, $F$ the general fiber of $f$, $m\geq1$ a rational number, and $D$ a Cartier divisor on $X$ such that $D\sim_{\Q}m(K_X+\Delta)$. Then
$$\kappa(X, K_X+\Delta)\ge\kappa(F, K_F+\Delta|_F)+\dim V^0(A, f_*\mathcal{O}_X(D)).$$
If $m>1$, then
$$\kappa(X, K_X+\Delta)\ge\kappa(F, K_F+\Delta|_F)+\kappa(A, \widehat{\det} f_*\mathcal{O}_X(D)).$$
\end{thm}

If $f_*\mathcal{O}_X(D)\neq0$, $V^0(A, f_*\mathcal{O}_X(D))$ is not empty since $f_*\mathcal{O}_X(D)$ is a GV-sheaf by \cite[Corollary 4.1]{Men21}. We have the following corollary of Theorem \ref{sskd}. If $\kappa(X, K_X+\Delta)=0$ and $f_*\mathcal{O}_X(D)\neq0$, then $\dim V^0(A, f_*\mathcal{O}_X(D))=0$ by Theorem \ref{sskd} and we deduce that $f_*\mathcal{O}_X(D)$ is a torsion line bundle on $A$ by employing its Chen--Jiang decomposition, see \cite[Theorem 5.2]{HPS18} for the case of pushforwards of pluricanonical bundles of smooth projective varieties under fibrations to abelian varieties.

By estimating the dimension of $V^0(A, f_*\mathcal{O}_X(D))$, we have the following corollary of Theorem \ref{sskd}.

\begin{coro}\label{sskdc}
Let $g\colon X\to Y$ be a smooth model of the Iitaka fibration associated to $K_X+\Delta$ with general fiber $G$ where $(X, \Delta)$ is a klt pair and $Y$ is a smooth projective variety, and $f\colon X \to A$ a fibration to an abelian variety $A$ with general fiber $F$. Then 
$$\kappa(X, K_X+\Delta)\ge\kappa(F, K_F+\Delta|_F)+\dim A-q(G).$$
\end{coro}

We can rewrite the inequality in Corollary \ref{sskdc} as
$$\dim F-\kappa(F, K_F+\Delta|_F)\ge\dim G-q(G)$$
where $\dim G-q(G)$ is nonnegative since the Albanese morphism of $G$ is a fibration by $\kappa(G, K_G+\Delta|_G)=0$ and \cite[Theorem B]{Wan16}. We immediately have the following corollary.

\begin{coro}\label{bta}
Let $g\colon X\to Y$ be a smooth model of the Iitaka fibration associated to $K_X+\Delta$ with general fiber $G$ where $(X, \Delta)$ is a klt pair and $Y$ is a smooth projective variety, and $f\colon X \to A$ a fibration to an abelian variety $A$ with general fiber $F$. If $(F, \Delta|_F)$ is of log general type, then $G$ is birational to its Albanese variety.
\end{coro}

Under the hypotheses of Corollary \ref{bta}, the klt pair $(G, \Delta|_G)$ has a good minimal model by \cite[Theorem 1.1]{Fuj13} since $G$ is birational to its Albanese variety. Thus the klt pair $(X, \Delta)$ has a good minimal model over $Y$ by \cite[Theorem 2.12]{HX13}. Since $g\colon X\to Y$ is a smooth model of the Iitaka fibration associated to $K_X+\Delta$, we can deduce that $(X, \Delta)$ has a good minimal model by running a $(K_X+\Delta)$-MMP over $Y$ and applying the canonical bundle formula. The main result of \cite{BC15} says that klt pairs fibered over normal projective varieties of maximal Albanese dimension with general fibers of log general type have good minimal models. By the discussion above, we give an intuitive explanation of why their result should be true.

We also have a quick corollary of Theorem \ref{sskd} if the Albanese morphism of $X$ is a fibration. In Corollary \ref{sskdc}, we let $f$ be the Albanese morphism of $X$. Then we deduce that
$$\kappa(X, K_X+\Delta)\ge\kappa(F, K_F+\Delta|_F)+q(Y)$$
by Theorem \ref{sskd} and \cite[Theorem 1.5]{Men21}.

For the proofs of the main theorems, we employ results from \cite{Men21}, techniques from \cite{MP21}, a hyperbolicity-type result from \cite{PS17}, and arguments on positivity properties of coherent sheaves.


\begin{ac}
{I would like to express my sincere gratitude to my advisor Mihnea Popa for helpful discussions and constant support. I would also like to thank Jungkai Alfred Chen for helpful discussions.}
\end{ac}

\section{Preliminaries}\label{2}

We work over $\C$. A \emph{fibration} is a projective surjective morphism with connected fibers. Let $\mathcal{F}$ be a coherent sheaf on a projective variety $X$, we denote $\mathcal{H}om_{\mathcal{O}_{X}}(\mathcal{F}, \mathcal{O}_{X})$ by $\mathcal{F}^{\vee}$.

We recall several definitions first.

\begin{defi}\label{irregularity}
Let $X$ be a smooth projective variety. The \emph{irregularity} $q(X)$ is defined as $h^1(X,\mathcal{O}_X)$. If $X$ is a projective variety, the \emph{irregularity} $q(X)$ is defined as the irregularity of any resolution of $X$.
\end{defi}

If $X$ is a normal projective variety of rational singularities, then the irregularity $q(X)$ is equal to the dimension of its Albanese variety $\Alb(X)$ since its Albanese variety coincides with the Albanese variety of any of its resolution by \cite[Proposition 2.3]{Rei83} and \cite[Lemma 8.1]{Kaw85a}.

\begin{defi}\label{csl}
Let $\mathcal{F}$ be a coherent sheaf on an abelian variety $A$. The \emph{cohomological support loci} $V_l^i(A, \mathcal{F})$ for $i\in\N$ and $l\in\N$ are defined by
$$V_l^i(A, \mathcal{F})=\{\alpha\in\Pic^0(A)\mid\dim H^i(A, \mathcal{F}\otimes\alpha)\geq l\}.$$
We use $V^i(A, \mathcal{F})$ to denote $V_1^i(A, \mathcal{F})$.
\end{defi}

\begin{defi}\label{GVM}
A coherent sheaf $\mathcal{F}$ on an abelian variety $A$
\begin{enumerate}
	\item[$\mathrm{(i)}$] is a GV-\emph{sheaf} if $\codim_{\Pic^0 (A)} V^i (A, \mathcal{F}) \ge i$ for every $i>0$.
	\item[$\mathrm{(ii)}$] is \emph{M-regular} if $\codim_{\Pic^0 (A)} V^i (A, \mathcal{F}) > i$ for every $i>0$.
	\item[$\mathrm{(iii)}$] \emph{satisfies} $\IT$ if $V^i (A, \mathcal{F})=\emptyset$ for every $i>0$.
\end{enumerate}
\end{defi}

It is known that M-regular sheaves are ample by \cite[Corollary 3.2]{Deb06}, and GV-sheaves are nef by \cite[Theorem 4.1]{PP11b}. We prove a useful lemma here by a similar method as in the proof of \cite[Theorem 4.1]{PP11b}.

\begin{lemma}\label{detn}
Let $\mathcal{F}$ be a torsion-free $\GV$-sheaf on an abelian variety $A$. Then $\widehat{\det}\mathcal{F}$ is nef.
\end{lemma}

\begin{proof}
We denote by $m_A\colon A\to A$ the multiplication by $m$ where $m$ is an integer. We can take an ample line bundle $\mathcal{H}$ on $A$ such that $(-1_A)^*\mathcal{H}\cong\mathcal{H}$ and thus we have $m_A^*\mathcal{H}\cong\mathcal{H}^{\otimes m^2}$. Since $\mathcal{F}$ is a $\GV$-sheaf and $m_A$ is an isogeny, $m_A^*\mathcal{F}$ is a torsion-free $\GV$-sheaf. We choose $m$ to be positive now. We deduce that $m_A^*\mathcal{F}\otimes\mathcal{H}^{\otimes m}$ satisfies $\IT$ by \cite[Proposition 3.1]{PP11b}, and it is ample by \cite[Corollary 3.2]{Deb06}. Since an ample sheaf is big and $m_A$ is an isogeny, we deduce that
$$\widehat{\det}(m_A^*\mathcal{F}\otimes\mathcal{H}^{\otimes m})\cong\widehat{\det}m_A^*\mathcal{F}\otimes\mathcal{H}^{\otimes m\cdot\rank\mathcal{F}}\cong m_A^*\widehat{\det}\mathcal{F}\otimes\mathcal{H}^{\otimes m\cdot\rank\mathcal{F}}$$
is big by \cite[Lemma 3.2(iii)]{Vie83b} (see also \cite[Properties 5.1.1]{Mor87}). We deduce that
$$m_A^*((\widehat{\det}\mathcal{F})^{\otimes m}\otimes\mathcal{H}^{\otimes\rank\mathcal{F}})\cong (m_A^*\widehat{\det}\mathcal{F}\otimes\mathcal{H}^{\otimes m\cdot\rank\mathcal{F}})^{\otimes m}$$
is big and thus ample since $A$ is an abelian variety. Thus $(\widehat{\det}\mathcal{F})^{\otimes m}\otimes\mathcal{H}^{\otimes\rank\mathcal{F}}$ is ample for every $m>0$ and we deduce that $\widehat{\det}\mathcal{F}$ is nef.
\end{proof}

We now give the definition of the Chen--Jiang decomposition.

\begin{defi}\label{CJD}
Let $\mathcal{F}$ be a coherent sheaf on an abelian variety $A$. The sheaf $\mathcal{F}$ is said to have the \emph{Chen--Jiang decomposition} if $\mathcal{F}$ admits a finite direct sum decomposition
$$\mathcal{F}\cong \bigoplus_{i\in I}(\alpha_i\otimes p_i^*\mathcal{F}_i),$$
where each $A_i$ is an abelian variety, each $p_i\colon A\to A_i$ is a fibration, each $\mathcal{F}_i$ is a nonzero M-regular coherent sheaf on $A_i$, and each $\alpha_i\in\Pic^0(A)$ is a torsion line bundle.
\end{defi}

We state the following theorem which is a direct corollary of \cite[Theorems 1.3 and 1.4]{Men21} and omit the proof, see also \cite[Theorem 1.3]{Jia21} for the case when $m>1$ is an integer.

\begin{thm}\label{IT0}
Let $f \colon X \to A$ be a morphism from a klt pair $(X, \Delta)$ to an abelian variety $A$, $m>1$ a rational number, and $D$ a Cartier divisor on $X$ such that $D\sim_{\Q}m(K_X+\Delta)$. Then there exists a fibration $p\colon A\to B$ to an abelian variety $B$ such that $f_*\mathcal{O}_X(lD)$ admits, for every positive integer $l$, a finite direct sum decomposition
$$f_*\mathcal{O}_X(lD)\cong\bigoplus_{i\in I}(\alpha_i\otimes p^*\mathcal{F}_i),$$
where each $\mathcal{F}_i$ is a nonzero coherent sheaf on $B$ satisfying $\IT$, and each $\alpha_i\in\Pic^0(A)$ is a torsion line bundle whose order can be bounded independently of $l$. If $g\colon X\to Y$ is a smooth model of the Iitaka fibration associated to the Cartier divisor $D$ with general fiber $G$ where $Y$ is a smooth projective variety, then 
$$\dim B\ge\dim A-q(G).$$
Moreover, if $f$ is surjective, then
$$q(Y)\ge\dim B.$$ 
\end{thm}

We will need the following hyperbolicity-type result proved in \cite{PS17}. It relies on important ideas and results of Viehweg--Zuo and Campana--P\u aun, and on the theory of Hodge modules.

\begin{thm}[{\cite[Theorem 4.1 and Theorem 3.5]{PS17}}]\label{thm:PS3}
Let $f \colon X \to Y$ be a fibration between smooth projective varieties where $Y$ is not uniruled. Assume that $f$ is smooth over the complement of a closed subset $Z \subseteq Y$, and there exists a positive integer $m$ such that $\widehat{\det} f_* \omega_{X/Y}^{\otimes m}$ is big. Denote by $D$ the union of the divisorial components of $Z$. Then the line bundle $\omega_Y (D)$ is big.
\end{thm}

The theorem above is stated in \cite{PS17} only when $Z = D$, but the proof shows more generally the statement above, since all the objects it involves can be constructed from $Y$ with any closed subset of codimension at least $2$ removed.

We include a useful lemma about the log Kodaira dimension on ambient varieties of nonnegative Kodaira dimension which is \cite[Lemma 2.6]{MP21}.

\begin{lemma}\label{lkd}
Let $X$ be a smooth projective variety with $\kappa(X)\ge 0$, $Z\subseteq X$ a closed reduced subscheme, and $V = X \setminus Z$. Assume that $Z=W\cup D$ where $\codim_X W\ge 2$ and $D$ is a divisor. Then
$$\kappa(V)=\kappa(X , K_X +D).$$
\end{lemma}

\section{Main results}\label{3}

We prove several useful lemmas first.

\begin{lemma}\label{fb}
Let $f \colon X \to Y$ be a surjective morphism between normal projective varieties, and $\varphi\colon Y'\to Y$ an \'etale morphism from a normal projective variety $Y'$. Consider the following base change diagram.
\[
\begin{tikzcd}
X' \dar{f'} \rar{\varphi'} & X \dar{f} \\
Y' \rar{\varphi} & Y
\end{tikzcd}
\]
Let $\mathcal{F}$ be a torsion-free coherent sheaf on $X$, then
$$\varphi^* \widehat{\det} f_*\mathcal{F} \cong \widehat{\det}\varphi^*f_* \mathcal{F}\cong \widehat{\det}f'_* \varphi'^*\mathcal{F}.$$
\end{lemma}

\begin{proof}
The coherent sheaves $f_*\mathcal{F}$ and $f'_* \varphi'^*\mathcal{F}$ are torsion-free since $\varphi'$ is \'etale. Since $\varphi$ is flat, we deduce that
$$\varphi^* \widehat{\det} f_*\mathcal{F} \cong \varphi^*\mathcal{H}om_{\mathcal{O}_Y}(\mathcal{H}om_{\mathcal{O}_Y}(\det f_*\mathcal{F}, \mathcal{O}_Y), \mathcal{O}_Y)$$
$$\cong\mathcal{H}om_{\mathcal{O}_{Y'}}(\mathcal{H}om_{\mathcal{O}_{Y'}}(\varphi^*\det f_*\mathcal{F}, \varphi^*\mathcal{O}_Y), \varphi^*\mathcal{O}_Y)$$
$$\cong\mathcal{H}om_{\mathcal{O}_{Y'}}(\mathcal{H}om_{\mathcal{O}_{Y'}}(\det \varphi^*f_* \mathcal{F}, \mathcal{O}_{Y'}), \mathcal{O}_{Y'})\cong\widehat{\det}\varphi^*f_* \mathcal{F}\cong \widehat{\det}f'_* \varphi'^*\mathcal{F}.$$
\end{proof}

\begin{lemma}\label{gb}
Let $f \colon X \to Y$ and $g \colon Y\to Z$ be surjective morphisms where $Y$ is a smooth projective variety, and $X$ and $Z$ are normal projective varieties. Consider the following base change diagram.
\[
\begin{tikzcd}
X_z \dar{f_z} \rar & X \dar{f} \\
Y_z \dar{g_z} \rar & Y \dar{g} \\
\{z\} \rar & Z 
\end{tikzcd}
\]
Let $\mathcal{F}$ be a locally free sheaf of finite rank on $X$. If $z$ is a general point of $Z$, then
$$(\widehat{\det} f_*\mathcal{F})|_{Y_z}\cong\widehat{\det} {f_z}_*(\mathcal{F}|_{X_z}).$$ 
\end{lemma}

\begin{proof}
Choose an open set $V\subseteq Y$ such that $(f_*\mathcal{F})|_V$ is locally free and $\codim_Y Y\setminus V\ge 2$. Consider the following base change diagram.
\[
\begin{tikzcd}
V_z \dar{i_z} \rar & V \dar{i} \\
Y_z \rar & Y
\end{tikzcd}
\]
We can choose $z$ sufficiently general such that $X_z$ is normal, $Y_z$ is smooth, $\codim_{Y_z} Y_z\setminus V_z\ge 2$, and
$$(f_*\mathcal{F})|_{Y_z}\cong{f_z}_*(\mathcal{F}|_{X_z})$$
by \cite[Proposition 4.1]{LPS20}. Thus $({f_z}_*(\mathcal{F}|_{X_z}))|_{V_z}$ is locally free. By the property of reflexive sheaves, we have
$$(\widehat{\det} f_*\mathcal{F})|_{Y_z}\cong (i_*\det i^*f_*\mathcal{F})|_{Y_z}\quad \text{and} \quad \widehat{\det} {f_z}_*(\mathcal{F}|_{X_z})\cong{i_z}_*\det i_z^*{f_z}_*(\mathcal{F}|_{X_z}).$$
We have the natural morphism
$$(\widehat{\det} f_*\mathcal{F})|_{Y_z}\cong(i_*\det i^*f_*\mathcal{F})|_{Y_z}\to {i_z}_*((\det i^*f_*\mathcal{F})|_{V_z})$$
$$\cong {i_z}_*i_z^*(\det(f_*\mathcal{F})|_{Y_z})\cong{i_z}_*i_z^*\det{f_z}_*(\mathcal{F}|_{X_z})\cong\widehat{\det} {f_z}_*(\mathcal{F}|_{X_z}).$$
The morphism above is an isomorphism over the open set $V_z$. Thus it is an isomorphism over $Y_z$ since $(\widehat{\det} f_*\mathcal{F})|_{Y_z}$ and $\widehat{\det} {f_z}_*(\mathcal{F}|_{X_z})$ are line bundles, and $\codim_{Y_z} Y_z\setminus V_z\ge 2$.
\end{proof}

\begin{lemma}\label{nz}
Let $f\colon X \to A$ be a surjective morphism from a klt pair $(X, \Delta)$ to an abelian variety $A$, $m\geq1$ a rational number, and $D$ a Cartier divisor on $X$ such that $D\sim_{\Q}m(K_X+\Delta)$. If $f_*\mathcal{O}_X(D)\neq0$, then
$$\kappa(A, \widehat{\det} f_*\mathcal{O}_X(D))\ge0.$$
\end{lemma}

\begin{proof}
By \cite[Theorem 1.1]{Men21}, there exists an isogeny $\varphi\colon A'\to A$ such that $\varphi^*f_*\mathcal{O}_X(D)$ is globally generated. We deduce that $\det \varphi^*f_*\mathcal{O}_X(D)$ is globally generated and thus $\widehat{\det} \varphi^*f_*\mathcal{O}_X(D)$ is generically globally generated. In particular, the line bundle $\widehat{\det}\varphi^*f_*\mathcal{O}_X(D)$ has nonzero sections. By Lemma \ref{fb}, we deduce
$$\kappa(A, \widehat{\det} f_*\mathcal{O}_X(D))=\kappa(A', \varphi^*\widehat{\det} f_*\mathcal{O}_X(D))=\kappa(A', \widehat{\det} \varphi^*f_*\mathcal{O}_X(D))\ge0.$$
\end{proof}

\begin{lemma}\label{po}
Let $f \colon X \to Y$ be a fibration between normal projective varieties, $F$ the very general fiber of $f$, and $\mathcal{L}$ a line bundle on $X$. If $f_*\mathcal{L}\cong\mathcal{B}\oplus\mathcal{T}$ where $\mathcal{B}$ is an ample sheaf on $Y$, then
$$\kappa(X, \mathcal{L})=\kappa(F, \mathcal{L}|_{F})+\dim Y.$$
\end{lemma}

\begin{proof}
Let $\mathcal{H}$ be an ample line bundle on $Y$. Since $\mathcal{B}$ is ample, there exists a positive integer $k$ such that $S^{k}\mathcal{B}\otimes\mathcal{H}^{-1}$ is globally generated where $S^{k}\mathcal{B}$ is the $k$-th symmetric product of $\mathcal{B}$ (see e.g. \cite[Section 2]{Deb06}). We have the following morphism
$$S^{k}\mathcal{B}\hookrightarrow S^{k}f_*\mathcal{L}\to f_*\mathcal{L}^{\otimes k}$$
which is the following nonzero multiplication homomorphism between $k(y)$-linear spaces when restricted at the general point $y$ of $Y$
$$S^{k}(\mathcal{B}_y\otimes_{\mathcal{O}_{Y, y}}k(y))\hookrightarrow S^{k}H^0(X_y, \mathcal{L}|_{X_y})\to H^0(X_y, \mathcal{L}^{\otimes k}|_{X_y})$$
by the base change theorem and generic flatness. Thus we deduce that the following homomorphism
$$H^0(Y, S^{k}\mathcal{B}\otimes\mathcal{H}^{-1})\to H^0(Y, f_*\mathcal{L}^{\otimes k}\otimes\mathcal{H}^{-1})$$
is nonzero since $S^{k}\mathcal{B}\otimes\mathcal{H}^{-1}$ is globally generated. We deduce that $f_*\mathcal{L}^{\otimes k}\otimes\mathcal{H}^{-1}$ has a nonzero global section and thus $\mathcal{L}^{\otimes k}\otimes (f^*\mathcal{H})^{-1}$ has a nonzero global section. Thus we have an injective morphism
$$f^*\mathcal{H}\to\mathcal{L}^{\otimes k}.$$
By \cite[Proposition 1.14]{Mor87}, we deduce that
$$\kappa(X, \mathcal{L})=\kappa(F, \mathcal{L}|_{F})+\dim Y.$$
\end{proof}

We are ready to prove our main theorems now. We prove the first part of the inequality in Theorem \ref{hp} first.

\begin{thm}\label{1}
Let $f \colon X \to A$ be a fibration from a smooth projective variety $X$ to an abelian variety $A$ where $f$ is smooth over an open set $V\subseteq A$, and $m$ a positive integer. Then 
$$\kappa(V)\ge\kappa(A, \widehat{\det} f_* \omega_{X}^{\otimes m}).$$
\end{thm}

\begin{proof}
If $f_* \omega_{X}^{\otimes m}=0$, then the statement is trivial. Thus we can assume $f_* \omega_{X}^{\otimes m}\neq0$. By Lemma \ref{nz}, we have that
$$\kappa(A, \widehat{\det} f_* \omega_{X}^{\otimes m})\ge0.$$
If $\kappa(A, \widehat{\det} f_* \omega_{X}^{\otimes m})=0$, then the statement is trivial since $\kappa(V)\ge0$. Thus we can assume $\kappa(A, \widehat{\det} f_* \omega_{X}^{\otimes m})>0$. Denote $Z = A \setminus V$ and assume that $Z=W\cup D$ where $\codim_A W\ge 2$ and $D$ is an effective divisor. If $\widehat{\det} f_* \omega_{X}^{\otimes m}$ is big, we deduce that
$$\kappa(V)=\kappa(A, K_A+D)=\dim A=\kappa(A, \widehat{\det} f_* \omega_{X}^{\otimes m})$$
by Theorem \ref{thm:PS3} and Lemma \ref{lkd}. Assume now $\widehat{\det} f_* \omega_{X}^{\otimes m}$ is not big. We can choose a positive integer $N$ which is sufficiently big and divisible such that
$$(\widehat{\det} f_* \omega_{X}^{\otimes m})^{\otimes N}\cong\mathcal{O}_A(E)$$ 
where $E$ is an effective divisor. By a well-known structural theorem, there exist a fibration $p\colon A \to B$ between abelian varieties and an ample effective divisor $H$ on $B$ such that $\dim A > \dim B >0$ and $E=p^*H$. Denote the kernel of $p$ by $K$ which is an abelian subvariety of $A$. By Poincar\'e's complete reducibility theorem, there exists an abelian variety $C\subseteq A$ such that $C+K=A$ and $C\cap K$ is finite, so that the natural morphism $\varphi\colon C\times K\to A$ is an isogeny. We consider the following commutative diagram, $q$ is the projection onto $K$, $k\in K$ is a general point, and $f'$ and $f'_{k}$ are obtained by base change from $f$ via $\varphi$ and the inclusion $i_k$ of the fiber $C_k$ of $q$ over $k$ respectively.
\[
\begin{tikzcd}
X'_k\rar \dar{f'_k} & X' \rar{\varphi'} \dar{f'} \rar & X \dar{f}  \\
C_k \dar \rar{i_k} & C\times K \dar{q} \rar{\varphi} & A \dar{p} \\
\{k\}\rar& K  & B
\end{tikzcd}
\]
By construction, the composition
$$\psi_k = p\circ\varphi\circ i_k \colon C_k \to B$$ 
is an isogeny. Since $\varphi$ is \'etale, $X'$ is smooth. If $W':= \varphi^{-1} (W)$, then $\codim_{C \times K}W'\ge2$, and $f'$ is smooth over $V':=\varphi^{-1} (V)$. We can choose $k$ sufficiently general such that $X'_{k}$ is smooth, $\codim_{C_k}  i_k^{-1}(W')\ge 2$, and
$$i_k^*\widehat{\det} f'_* \omega_{X'}^{\otimes m}\cong\widehat{\det} {f'_k}_*(\omega_{X'}^{\otimes m}|_{X'_k})\cong \widehat{\det} {f'_k}_*\omega_{X'_k}^{\otimes m}$$
by Lemma \ref{gb}. By Lemma \ref{fb}, we deduce that
$$\psi_k^*\mathcal{O}_B(H)\cong i_k^*\varphi^*\mathcal{O}_A(E)\cong i_k^*\varphi^*(\widehat{\det} f_* \omega_{X}^{\otimes m})^{\otimes N}$$
$$\cong i_k^*(\widehat{\det} f'_* \omega_{X'}^{\otimes m})^{\otimes N}\cong(\widehat{\det} {f'_k}_*\omega_{X'_k}^{\otimes m})^{\otimes N}.$$
Since $\psi_k$ is an isogeny, $\psi_k^*\mathcal{O}_B(H)$ is ample and thus $\widehat{\det} {f'_k}_*\omega_{X'_k}^{\otimes m}$ is ample. Since $\codim_{C_k}  i_k^{-1}(W')\ge 2$ and $f'_k$ is smooth over $i_k^{-1}(V')$, we deduce that
$$\kappa(C_k, K_{C_k}+ (i_k^*\varphi^*D)_{\red})=\dim C_k=\dim B=\kappa(A, E)=\kappa(A, \widehat{\det} f_* \omega_{X}^{\otimes m})$$
by Theorem \ref{thm:PS3}. We can choose a rational number $\varepsilon>0$ small enough such that $(C\times K, \varepsilon\varphi^*D)$ is a klt pair. By \cite[Theorem 1.1]{CP17} and choosing $k$ sufficiently general, we deduce that
$$\kappa(V)=\kappa(A, D)=\kappa(A, \varepsilon D)=\kappa(C\times K, \varepsilon\varphi^*D)$$
$$\ge\kappa(C_k, \varepsilon i_k^*\varphi^*D)=\kappa(C_k, (i_k^*\varphi^*D)_{\red})=\kappa(A, \widehat{\det} f_* \omega_{X}^{\otimes m}).$$
\end{proof}

Next, we prove the second part of the inequality in Theorem \ref{hp} in a more general setting which allows klt singularities. We also prove the first inequality in Theorem \ref{sskd} along the way.

\begin{thm}\label{2}
Let $f\colon X \to A$ be a fibration from a klt pair $(X, \Delta)$ to an abelian variety $A$, $F$ the general fiber of $f$, $m\geq1$ a rational number, and $D$ a Cartier divisor on $X$ such that $D\sim_{\Q}m(K_X+\Delta)$. Then
$$\kappa(X, K_X+\Delta)\ge\kappa(F, K_F+\Delta|_F)+\dim V^0(A, f_*\mathcal{O}_X(D)),$$
and
$$\kappa(A, \widehat{\det} f_*\mathcal{O}_X(D))\ge\dim V^0(A, f_*\mathcal{O}_X(D)).$$
\end{thm}

\begin{proof}
If $f_*\mathcal{O}_X(D)=0$, then the statement is trivial. Thus we can assume $f_*\mathcal{O}_X(D)\neq0$. By \cite[Theorem 1.3]{Men21}, $f_*\mathcal{O}_X(D)$ has the Chen--Jiang decomposition
$$f_*\mathcal{O}_X(D)\cong \bigoplus_{i\in I}(\alpha_i\otimes p_i^*\mathcal{F}_i),$$
where each $A_i$ is an abelian variety, each $p_i\colon A\to A_i$ is a fibration, each $\mathcal{F}_i$ is a nonzero M-regular coherent sheaf on $A_i$, and each $\alpha_i\in\Pic^0(A)$ is a torsion line bundle. By \cite[Lemma 3.3]{LPS20}, we deduce that
$$\dim V^0(A, f_*\mathcal{O}_X(D))=\max_{i\in I} \dim A_i.$$
We consider the fibration $p_j\colon A\to A_j$ for a fixed $j\in I$. Denote the kernel of $p_j$ by $K$ which is an abelian subvariety of $A$. By Poincar\'e's complete reducibility theorem, there exists an abelian variety $C\subseteq A$ such that $C+K=A$ and $C\cap K$ is finite, so that the natural morphism $\varphi\colon C\times K\to A$ is an isogeny. We consider the following commutative diagram, $q$ is the projection onto $K$, $k\in K$ is a general point, and $f'$ and $f'_{k}$ are obtained by base change from $f$ via $\varphi$ and the inclusion $i_k$ of the fiber $C_k$ of $q$ over $k$ respectively. We define a $\Q$-divisor $\Delta'$ by $K_{X'}+\Delta'=\varphi'^*(K_X+\Delta)$. Since $\varphi'$ is an \'etale morphism, the new pair $(X', \Delta')$ is klt and $\Delta'$ is effective. Define $D'$ by $\varphi'^*D$ then we have $D'\sim_{\Q}m(K_{X'}+\Delta')$. By the flat base change theorem, we have that $f'_*\mathcal{O}_{X'}(D')\cong\varphi^*f_*\mathcal{O}_X(D)$.  
\[
\begin{tikzcd}
(X'_k, \Delta'|_{X'_k})\rar \dar{f'_k} & (X', \Delta') \rar{\varphi'} \dar{f'} \rar & (X, \Delta) \dar{f}  \\
C_k \dar \rar{i_k} & C\times K \dar{q} \rar{\varphi} & A \dar{p_j} \\
\{k\}\rar& K  & A_j
\end{tikzcd}
\]
By construction, the composition
$$\psi_k = p_j\circ\varphi\circ i_k \colon C_k \to A_j$$ 
is an isogeny. If we denote $\mathcal{T} : = \bigoplus_{i\neq j}(\alpha_i\otimes p_i^*\mathcal{F}_i)$, then we have 
$$f_*\mathcal{O}_X(D)\cong (\alpha_j\otimes p_j^*\mathcal{F}_j) \oplus \mathcal{T}.$$
We can choose $k$ sufficiently general such that $(X'_k, \Delta'|_{X'_k})$ is klt and
$${f'_k}_*\mathcal{O}_{X'_k}(D'|_{X'_k})\cong i_k^*f'_*\mathcal{O}_{X'}(D')\cong i_k^*\varphi^*f_*\mathcal{O}_X(D)$$
by \cite[Proposition 4.1]{LPS20}. We deduce that
$${f'_k}_*\mathcal{O}_{X'_k}(D'|_{X'_k})\cong i_k^*\varphi^*f_*\mathcal{O}_X(D)\cong(i_k^*\varphi^*\alpha_j\otimes \psi_k^*\mathcal{F}_j) \oplus i_k^*\varphi^*\mathcal{T}.$$
Since $\mathcal{F}_j$ is an M-regular sheaf on $A_j$, it is ample by \cite[Proposition 2.13]{PP03} and \cite[Corollary 3.2]{Deb06}. Since $\psi_k$ is an isogeny and $i_k^*\varphi^*\alpha_j$ is a torsion line bundle, we deduce that $i_k^*\varphi^*\alpha_j\otimes \psi_k^*\mathcal{F}_j$ is also ample. We deduce that $i_k^*\varphi^*\mathcal{T}$ is a GV-sheaf since it is a direct summand of ${f'_k}_*\mathcal{O}_{X'_k}(D'|_{X'_k})$ which is a GV-sheaf by \cite[Corollary 4.1]{Men21}. Thus we have that
$${f'_k}_*\mathcal{O}_{X'_k}(D'|_{X'_k})\cong\mathcal{H}_1 \oplus \mathcal{H}_2,$$
with $\mathcal{H}_1$ ample and $\mathcal{H}_2$ a GV-sheaf. The very general fiber of $f'_k$ is $F$. We deduce that
$$\kappa(X'_k, D'|_{X'_k})=\kappa(F, D|_{F})+\dim C_k=\kappa(F, D|_{F})+\dim A_j$$
by Lemma \ref{po}. By \cite[Theorem 4.2]{HMX18}, $\kappa(F, K_F+\Delta|_F)$ is constant for general fiber $F$ of $f$. By \cite[Theorem 1.1]{CP17} and choosing $k$ sufficiently general, we deduce that
$$\kappa(X, K_X+\Delta)=\kappa(X', K_{X'}+\Delta')\ge\kappa(X'_k, K_{X'_k}+\Delta'|_{X'_k})$$
$$=\kappa(X'_k, D'|_{X'_k})=\kappa(F, D|_{F})+\dim A_j=\kappa(F, K_F+\Delta|_F)+\dim A_j.$$
Thus we deduce that
\begin{align*}
\kappa(X, K_X+\Delta)&\ge\kappa(F, K_F+\Delta|_F)+\max_{i\in I} \dim A_i\\
&=\kappa(F, K_F+\Delta|_F)+\dim V^0(A, f_*\mathcal{O}_X(D)).
\end{align*}

We next prove the second inequality in the theorem. Since $\mathcal{H}_1$ is big, $\widehat{\det}\mathcal{H}_1$ is a big line bundle by \cite[Lemma 3.2(iii)]{Vie83b} (see also \cite[Properties 5.1.1]{Mor87}). We deduce that $\widehat{\det}\mathcal{H}_2$ is a nef line bundle by Lemma \ref{detn}. Thus their tensor product $\widehat{\det}{f'_k}_*\mathcal{O}_{X'_k}(D'|_{X'_k})$ is big. We can choose $k$ sufficiently general such that
$$i_k^*\widehat{\det}f'_*\mathcal{O}_{X'}(D')\cong \widehat{\det} {f'_k}_*\mathcal{O}_{X'_k}(D'|_{X'_k})$$
by Lemma \ref{gb}. By Lemma \ref{nz}, we can choose a positive integer $N$ which is sufficiently big and divisible such that
$$(\widehat{\det} f_*\mathcal{O}_X(D))^{\otimes N}\cong\mathcal{O}_A(E)$$ 
where $E$ is an effective divisor. By Lemma \ref{fb}, we deduce that
$$i_k^*\varphi^*\mathcal{O}_A(E)\cong i_k^*\varphi^*(\widehat{\det}f_*\mathcal{O}_X(D))^{\otimes N}$$
$$\cong i_k^*(\widehat{\det} f'_*\mathcal{O}_{X'}(D'))^{\otimes N}\cong(\widehat{\det} {f'_k}_*\mathcal{O}_{X'_k}(D'|_{X'_k}))^{\otimes N}.$$
By the same argument used at the end of the proof of Theorem \ref{1} and choosing $k$ sufficiently general, we deduce that
$$\kappa(A, \widehat{\det} f_*\mathcal{O}_X(D))=\kappa(A, \mathcal{O}_A(E))\ge\kappa(C_k, i_k^*\varphi^*\mathcal{O}_A(E))$$
$$=\kappa(C_k, \widehat{\det} {f'_k}_*\mathcal{O}_{X'_k}(D'|_{X'_k}))=\dim C_k=\dim A_j.$$
Thus we deduce that
$$\kappa(A, \widehat{\det} f_*\mathcal{O}_X(D))\ge\max_{i\in I} \dim A_i=\dim V^0(A, f_*\mathcal{O}_X(D)).$$
\end{proof}

\begin{rem}
We still have $\kappa(A, \widehat{\det} f_*\mathcal{O}_X(D))\ge\dim V^0(A, f_*\mathcal{O}_X(D))$ if $f$ is only assumed to be surjective by the same proof as above.
\end{rem}

\begin{lemma}\label{equality}
Let $f\colon X \to A$ be a surjective morphism from a klt pair $(X, \Delta)$ to an abelian variety $A$, $m>1$ a rational number, and $D$ a Cartier divisor on $X$ such that $D\sim_{\Q}m(K_X+\Delta)$. Then
$$\kappa(A, \widehat{\det} f_*\mathcal{O}_X(D))=\dim V^0(A, f_*\mathcal{O}_X(D)).$$
\end{lemma}

\begin{proof}
If $f_*\mathcal{O}_X(D)=0$, then the statement is trivial. Thus we can assume $f_*\mathcal{O}_X(D)\neq0$. By Theorem \ref{IT0}, there exists a fibration $p\colon A\to B$ to an abelian variety $B$ such that $f_*\mathcal{O}_X(D)$ admits a finite direct sum decomposition
$$f_*\mathcal{O}_X(D)\cong\bigoplus_{i\in I}(\alpha_i\otimes p^*\mathcal{F}_i),$$
where each $\mathcal{F}_i$ is a nonzero coherent sheaf on $B$ satisfying $\IT$, and each $\alpha_i\in\Pic^0(A)$ is a torsion line bundle. Each $\mathcal{F}_i$ is torsion-free. By \cite[Lemma 3.3]{LPS20}, we deduce that
$$\dim V^0(A, f_*\mathcal{O}_X(D))=\dim B.$$
Since $p$ is flat and $\widehat{\det}(\alpha_i\otimes p^*\mathcal{F}_i)$ is a line bundle, we deduce that
$$\widehat{\det} f_*\mathcal{O}_X(D)\cong\bigotimes_{i\in I}\widehat{\det}(\alpha_i\otimes p^*\mathcal{F}_i)\cong\bigotimes_{i\in I}(\widehat{\det}p^*\mathcal{F}_i\otimes\alpha_i^{\otimes\rank\mathcal{F}_i})$$
$$\cong\bigotimes_{i\in I}(p^*\widehat{\det}\mathcal{F}_i\otimes\alpha_i^{\otimes\rank\mathcal{F}_i})\cong p^*(\bigotimes_{i\in I}\widehat{\det}\mathcal{F}_i)\otimes\bigotimes_{i\in I}\alpha_i^{\otimes\rank\mathcal{F}_i}.$$
Since $\mathcal{F}_i$ satisfies $\IT$, it is ample by \cite[Proposition 2.13]{PP03} and \cite[Corollary 3.2]{Deb06}. The sheaf $\mathcal{F}_i$ is big since an ample sheaf is big (see e.g. \cite[Section 2]{Deb06} and \cite[Section 5]{Mor87}). Thus $\widehat{\det}\mathcal{F}_i$ is a big line bundle by \cite[Lemma 3.2(iii)]{Vie83b} (see also \cite[Properties 5.1.1]{Mor87}). We deduce that
$$\kappa(A, \widehat{\det} f_*\mathcal{O}_X(D))=\kappa(A, p^*(\bigotimes_{i\in I}\widehat{\det}\mathcal{F}_i)\otimes\bigotimes_{i\in I}\alpha_i^{\otimes\rank\mathcal{F}_i})$$
$$=\kappa(A, p^*(\bigotimes_{i\in I}\widehat{\det}\mathcal{F}_i))=\kappa(B, \bigotimes_{i\in I}\widehat{\det}\mathcal{F}_i)=\dim B=\dim V^0(A, f_*\mathcal{O}_X(D)).$$
\end{proof}

\begin{proof}[Proof of Theorem \ref{hp}]
It is a direct corollary of Theorem \ref{1}, Theorem \ref{2} and Lemma \ref{equality}.
\end{proof}

\begin{proof}[Proof of Corollary \ref{hpc}]
It is a direct corollary of Theorems \ref{hp} and \ref{IT0}.
\end{proof}

\begin{proof}[Proof of Corollary \ref{nns}]
Let $f\colon X \to A$ be a smooth morphism to an abelian variety $A$. Then $f$ is surjective. We consider its Stein factorization $f=\varphi\circ h$ where $B$ is a normal projective variety, $h\colon X\to B$ is a fibration, and $\varphi\colon B\to A$ is a finite surjective morphism. Since $f$ is smooth, we deduce that $h$ is smooth and $\varphi$ is \'etale. Thus $B$ is also an abelian variety. Since $q(G)=0$ and $h$ is a smooth fibration, we deduce that $\dim A=\dim B=0$ by Corollary \ref{hpc}.
\end{proof}

\begin{proof}[Proof of Corollary \ref{var}]
By Theorem \ref{IT0} and \cite[Lemma 3.3]{LPS20}, there exists an abelian variety $B$ such that
$$\dim V^0(A, f_*\omega_X^{\otimes m})=\dim B$$
for every integer $m>1$ such that $f_*\omega_X^{\otimes m}\neq0$. By Theorem \ref{hp}, we have
$$\kappa(A, \widehat{\det} f_* \omega_{X}^{\otimes m})=\dim V^0(A, f_*\omega_X^{\otimes m})$$
for every integer $m>1$. By \cite[Theorem 1.1]{Kaw85a}, there exists an integer $k>1$ such that
$$\kappa(A, \widehat{\det} f_* \omega_{X}^{\otimes k})\ge\Var(f)$$
since the general fiber of $f$ has a good minimal model. In particular, $f_*\omega_X^{\otimes k}\neq0$. Thus we deduce that
$$\dim V^0(A, f_*\omega_X^{\otimes m})=\dim V^0(A, f_*\omega_X^{\otimes k})=\kappa(A, \widehat{\det} f_* \omega_{X}^{\otimes k})\ge\Var(f)$$
for every integer $m>1$ such that $f_*\omega_X^{\otimes m}\neq0$. If $g\colon X\to Y$ is a smooth model of the Iitaka fibration of $X$ where $Y$ is a smooth projective variety, then
$$q(Y)\ge\dim B\ge\Var(f)$$
by Theorem \ref{IT0}.
\end{proof}

\begin{proof}[Proof of Theorem \ref{sskd}]
It is a direct corollary of Theorem \ref{2} and Lemma \ref{equality}.
\end{proof}

\begin{proof}[Proof of Corollary \ref{sskdc}]
It is a direct corollary of Theorems \ref{sskd} and \ref{IT0}.
\end{proof}

\begin{proof}[Proof of Corollary \ref{bta}]
It is a direct corollary of Corollary \ref{sskdc} and \cite[Theorem B]{Wan16}.
\end{proof}

	\bibliographystyle{amsalpha}
	\bibliography{biblio}	

\providecommand{\bysame}{\leavevmode\hbox to3em{\hrulefill}\thinspace}
\providecommand{\MR}{\relax\ifhmode\unskip\space\fi MR }
\providecommand{\MRhref}[2]{%
  \href{http://www.ams.org/mathscinet-getitem?mr=#1}{#2}
}
\providecommand{\href}[2]{#2}
\begin{thebibliography}{HMX18}

\bibitem[BC15]{BC15}
C.~Birkar and J.~A. Chen, \emph{Varieties fibred over abelian varieties with
  fibres of log general type}, Adv. Math. \textbf{270} (2015), 206--222.

\bibitem[CJ18]{CJ18}
J.~A. Chen and Z.~Jiang, \emph{{Positivity in varieties of maximal Albanese
  dimension}}, J. Reine Angew. Math. \textbf{736} (2018), 225--253.

\bibitem[CP17]{CP17}
J.~Cao and M.~P\u{a}un, \emph{Kodaira dimension of algebraic fiber spaces over
  abelian varieties}, Invent. Math. \textbf{207} (2017), no.~1, 345--387.

\bibitem[Deb06]{Deb06}
O.~Debarre, \emph{On coverings of simple abelian varieties}, Bull. Soc. Math.
  France \textbf{134} (2006), no.~2, 253--260.

\bibitem[Fuj13]{Fuj13}
O.~Fujino, \emph{{On maximal Albanese dimensional varieties}}, Proc. Japan
  Acad. Ser. A Math. Sci. \textbf{89} (2013), no.~8, 92--95.

\bibitem[HK05]{HK05}
C.~D. Hacon and S.~J. Kov\'{a}cs, \emph{{Holomorphic one-forms on varieties of
  general type}}, Ann. Sci. \'{E}cole Norm. Sup. \textbf{38} (2005), no.~4,
  599--607.

\bibitem[HMX18]{HMX18}
C.~D. Hacon, J.~M\textsuperscript{c}Kernan, and C.~Xu, \emph{{Boundedness of
  moduli of varieties of general type}}, J. Eur. Math. Soc. \textbf{20} (2018),
  no.~4, 865--901.

\bibitem[HPS18]{HPS18}
C.~D. Hacon, M.~Popa, and C.~Schnell, \emph{{Algebraic fiber spaces over
  abelian varieties: around a recent theorem by Cao and P{\u{a}}un}}, Local and
  global methods in Algebraic Geometry: volume in honor of L. Ein's 60th
  birthday, Contemp. Math. \textbf{712} (2018), 143--195.

\bibitem[HX13]{HX13}
C.~D. Hacon and C.~Xu, \emph{Existence of log canonical closures}, Invent.
  Math. \textbf{192} (2013), 161--195.

\bibitem[Jia21]{Jia21}
Z.~Jiang, \emph{{M-regular decompositions for pushforwards of pluricanonical
  bundles of pairs to abelian varieties}}, Int. Math. Res. Not. (2021).

\bibitem[Kaw85]{Kaw85a}
Y.~Kawamata, \emph{{Minimal models and the Kodaira dimension of algebraic fiber
  spaces}}, J. Reine Angew. Math. \textbf{363} (1985), 1--46.

\bibitem[LPS20]{LPS20}
L.~Lombardi, M.~Popa, and C.~Schnell, \emph{{Pushforwards of pluricanonical
  bundles under morphisms to abelian varieties}}, J. Eur. Math. Soc.
  \textbf{22} (2020), no.~8, 2511--2536.

\bibitem[Men21]{Men21}
F.~Meng, \emph{{Pushforwards of klt pairs under morphisms to abelian
  varieties}}, Math. Ann. \textbf{380} (2021), no.~3, 1655--1685.

\bibitem[Mor87]{Mor87}
S.~Mori, \emph{Classification of higher-dimensional varieties}, Algebraic
  geometry, {B}owdoin, 1985 ({B}runswick, {M}aine, 1985), Proc. Sympos. Pure
  Math., vol.~46, Amer. Math. Soc., Providence, RI, 1987, pp.~269--331.

\bibitem[MP21]{MP21}
F.~Meng and M.~Popa, \emph{{Kodaira dimension of fibrations over abelian
  varieties}}, preprint, arXiv:2111.14165\setbox0=\hbox{2021}.

\bibitem[Pop22]{Pop22}
M.~Popa, \emph{private communication}, 2022.

\bibitem[PP03]{PP03}
G.~Pareschi and M.~Popa, \emph{{Regularity on abelian varieties I}}, J. Amer.
  Math. Soc. \textbf{16} (2003), no.~2, 285--302.

\bibitem[PP11]{PP11b}
\bysame, \emph{{Regularity on abelian varieties III: relationship with generic
  vanishing and applications}}, Grassmannians, moduli spaces and vector
  bundles, Clay Math. Proc., vol.~14, Amer. Math. Soc., Providence, RI, 2011,
  pp.~141--167.

\bibitem[PPS17]{PPS17}
G.~Pareschi, M.~Popa, and C.~Schnell, \emph{Hodge modules on complex tori and
  generic vanishing for compact {K}{\"a}hler manifolds}, Geometry and Topology
  \textbf{21} (2017), no.~4, 2419--2460.

\bibitem[PS14]{PS14a}
M.~Popa and C.~Schnell, \emph{{Kodaira dimension and zeros of holomorphic
  one-forms}}, Ann. of Math. \textbf{179} (2014), no.~3, 1109--1120.

\bibitem[PS17]{PS17}
\bysame, \emph{{Viehweg's hyperbolicity conjecture for families with maximal
  variation}}, Invent. Math. \textbf{208} (2017), no.~3, 677--713.

\bibitem[Rei83]{Rei83}
M.~Reid, \emph{{Projective morphisms according to Kawamata}}, preprint, 1983.

\bibitem[Taj20]{Taj20}
B.~Taji, \emph{{Birational geometry of smooth families of varieties admitting
  good minimal models}}, preprint, arXiv:2005.01025\setbox0=\hbox{2020}.

\bibitem[Vie83]{Vie83b}
E.~Viehweg, \emph{{Weak positivity and the additivity of the Kodaira dimension.
  II. The local Torelli map}}, Classification of algebraic and analytic
  manifolds (Katata, 1982), Progr. Math., Birkh\"auser Boston, \textbf{39}
  (1983), 567--589.

\bibitem[VZ01]{VZ01}
E.~Viehweg and K.~Zuo, \emph{{On the isotriviality of families of projective
  manifolds over curves}}, J. Algebraic Geom. \textbf{10} (2001), 781--799.

\bibitem[Wan16]{Wan16}
Y.~Wang, \emph{{On the characterization of abelian varieties for log pairs in
  zero and positive characteristic}}, preprint,
  arXiv:1610.05630\setbox0=\hbox{2016}.

\end{thebibliography}

\end{document}